\documentclass[12pt]{amsart}
\usepackage{graphicx}
\usepackage{amsmath, amssymb, mathtools}
\usepackage{amsfonts}
\usepackage{bbm,bm}
\usepackage{hyperref}
\usepackage[margin=1in]{geometry}
\hypersetup{
    colorlinks=true,
    linkcolor=black,
    citecolor=black
}

\newtheorem{theorem}{Theorem}[section]

\newtheorem{proposition}[theorem]{Proposition}
\newtheorem{corollary}{Corollary}[theorem]
\newtheorem{assumption}[theorem]{Assumption}

\theoremstyle{definition}
\newtheorem{definition}[theorem]{Definition}

\theoremstyle{remark}
\newtheorem{remark}{Remark}

\newcommand{\ms}[1]{(\mathbf{#1},\mathcal{#1})}

\title{The Variational Approach in Filtering and Correlated Noise}

\author{Sharan Srinivasan}
\address[S.\ Srinivasan]{Elmore Family School of Electrical and Computer Engineering, Purdue University, West Lafayette, IN, USA}
\email{srini256@purdue.edu}

\author{Vijay Gupta}
\address[V.\ Gupta]{Elmore Family School of Electrical and Computer Engineering, Purdue University, West Lafayette, IN, USA}
\email{gupta869@purdue.edu}

\author{Harsha Honnappa}
\address[H.\ Honnappa]{Edwardson School of Industrial Engineering, Purdue University, West Lafayette, IN, USA}
\email{honnappa@purdue.edu}

\begin{document}

\begin{abstract}
The variational formulation of nonlinear filtering due to Mitter and Newton characterizes the filtering distribution as the unique minimizer of a free energy functional involving the relative entropy with respect to the prior and an expected energy. This formulation rests on an absolute continuity condition between the joint path measure and a product reference measure. We prove that this condition necessarily fails whenever the signal and observation diffusions share a common noise source. Specifically we show that the joint and product measures are mutually singular, so no choice of reference measure can salvage the formulation. We then introduce a conditional variational principle that replaces the prior with a reference measure that preserves the noise correlation structure. This generalization recovers the Mitter--Newton formulation as a special case when the noises are independent, and yields an explicit free energy characterization of the filter in the linear correlated-noise setting.
\end{abstract}

\maketitle

\section{Introduction}

The problem of estimating the state of a stochastic system from partial, noisy observations is among the most fundamental in applied probability and control theory. In the linear Gaussian setting, the Kalman–Bucy filter provides an elegant, finite-dimensional recursive solution. For nonlinear systems, however, the filtering distribution is generally infinite-dimensional, and characterizing it requires more sophisticated tools.

In a landmark paper, Mitter and Newton~\cite{mitter2003variational} formulated a variational approach to this problem, showing that the filtering distribution can be characterized as a Gibbs measure and is the unique minimizer of a free energy functional, the sum of the relative entropy with respect to the prior and an expected energy term. A key application of their framework is to nonlinear filtering of diffusions, where the signal ($X$) and observation ($Y$) processes are driven by independent Brownian motions. This independence of the Brownian motions is essential, as it permits the use of the Cameron–Martin–Girsanov theorem to establish that the joint path measure $P_{XY}$ is absolutely continuous with respect to the product $P_X \otimes \lambda_Y$ for a suitable reference measure $\lambda_Y$
 (Assumption~\ref{as: H1} below). The entire variational formulation rests on this absolute continuity condition. 
 
 On the other hand, filtering with correlated noise has been extensively studied, primarily through the Kushner–Stratonovich and Zakai equations and their robust formulations; see~\cite{clark1978design, elliott1981robust, davis1987pathwise, crisanrough}. These models have also been extended to systems driven by L\'evy noise~\cite{qiaouncorr, qiaocorr} and, more recently, to rough-path frameworks~\cite{allan2025rough,srinivasan2026robust}. In this paper, we show that Mitter and Newton's condition that $P_{XY} \ll P_x \otimes \lambda_Y$ fails, and cannot be rescued, when the signal and observation diffusions share a common noise source. We first present a discrete-time system with correlated noise to build intuition, showing that the Radon–Nikodym derivative between the joint measure and the product measure degenerates in the continuum limit. We then prove that Assumption~\ref{as: H1} can \emph{only} be satisfied when $P_{XY} \ll P_X \otimes P_Y$, so that no choice of reference measure $\lambda_Y$
 can salvage the formulation. In the linear Gaussian case, we establish mutual singularity of the joint and product measures via the Feldman–H\'ajek theorem, which provides a clean dichotomy since Gaussian measures on infinite-dimensional spaces are either equivalent or mutually singular. For the general nonlinear setting, we construct an explicit set that has full measure under $P_{XY}$ and null measure under $P_X \otimes P_Y$, establishing mutual singularity directly. Finally, we present a variational formulation that strictly generalizes~\cite{mitter2003variational} to the correlated noise setting by replacing the prior with a conditional reference measure that preserves the noise coupling structure.

\section{The Variational Formulation}

We first recall the variational approach defined in Section 2 of \cite{mitter2003variational}. Let $(\Omega,\mathcal{F},\mathbb{P})$ be a probability space, and let $\ms{X}$ and $\ms{Y}$ be Borel spaces. Let $X:\Omega\to\mathbf{X}$ and $Y:\Omega\to\mathbf{Y}$ be two measurable mappings with distributions $P_X$ and $P_Y$. Their joint distribution on $\mathcal{X\times Y}$ is $P_{XY}$.
\begin{assumption}\label{as: H1}
    There exists a $\sigma$-finite measure $\lambda_Y$ on $\mathcal{Y}$ such that
    \begin{equation*}
        P_{XY}\ll P_X\otimes\lambda_Y
    \end{equation*}
\end{assumption}

We now define a few sets and quantities:
\begin{definition}
    Define the map $Q:\mathbf{X\times Y}\to[0,\infty)$ as:
    \[
       Q(x,y) := \frac{dP_{XY}}{d(P_X\otimes\lambda_Y)}(x,y),
    \]
    and the set
    \[
        \bar{\mathbf{Y}} := \left\{y\in\mathbf{Y}\vert 0<\int_\mathbf{X}Q(x,y)P_X(dx)<\infty\right\}.
    \]
\end{definition}
We have $\bar{\mathbf{Y}}\in\mathcal{Y}$ and $P_Y(\mathbf{\bar{Y}})=1$.
\begin{definition}
    Let $H:\mathbf{X\times Y}\to(-\infty,\infty]$ be defined by
    \begin{equation*}
        H(x,y) :=\begin{cases}
            -\log(Q(x,y))\quad \text{ if } y\in\mathbf{\bar{Y}},\\
            0 \quad \text{otherwise};
        \end{cases}
    \end{equation*}
\end{definition}
then the Gibbs measure $P_{X|Y}:\mathcal{X}\times\mathbf{Y}\to[0,1]$ defined by
\begin{equation}\label{eq:gibbs form of conditional}
    P_{X|Y}(A,y):=\frac{\int_A\exp(-H(x,y))P_X(dx)}{\int_\mathbf{X}\exp(-H(x,y))P_X(dx)},
\end{equation}
is the conditional probability measure of $X$ given $Y$.

The key proposition \cite[Proposition 2.1]{mitter2003variational} (Proposition \ref{prop:gibbs var}) characterizes $P_{X|Y}$ as the minimizer of a free energy:

\begin{proposition}\label{prop:gibbs var}
    Suppose Assumption \ref{as: H1} is satisfied, then for any $y\in\mathbf{Y}$ such that
    \begin{align*}
        -\int_\mathbf{X}H(x,y)\exp(-H(x,y))P_X(dx)&<\infty \quad (\textnormal{where } +\infty\exp(-\infty)=0),
    \end{align*}
    $P_{X|Y}(\cdot,y)$ is the unique minimizer of
    \begin{equation}\label{eq:free energy}
        \min_{\tilde{P}_X\in\mathcal{P}(\mathcal{X})}\bigl\{h(\tilde{P}_X\vert P_X)+\langle H(\cdot,y),\tilde{P}_X\rangle\bigr\},
    \end{equation}
    where $h(P\vert Q)$ is the relative entropy of the measures $P,Q\in\mathcal{P}(\mathcal{X})$, and
    \begin{equation*}
        \langle H(\cdot,y),\tilde{P}_X\rangle := \begin{cases}
            \int_\mathbf{X}H(x,y)\tilde{P}_X(dx) \,\, \textnormal{if the integral exists},\\
            +\infty \quad \textnormal{otherwise}.
        \end{cases}
    \end{equation*}
\end{proposition}
This follows directly from the Gibbs variational principle and that the conditional measure $P_{X|Y}$ defined in equation (\ref{eq:gibbs form of conditional}) is a Gibbs measure.

\subsection{Path Measures of Diffusion Processes}

The Proposition \ref{prop:gibbs var} can be applied to the filtering setting to obtain the filtering measure. Let $X_.\in C([0,T],\mathbb{R}^d)$ and $Y_.\in C([0,T],\mathbb{R}^n)$ be processes satisfying the following It\^o SDEs:
\begin{align}\label{eq:sig uncorr}
    X_t &= X_0 + \int_0^tb(s,X_s)ds + \int_0^t\sigma(s,X_s)dB_s,\\
    \label{eq:obs uncorr}
    Y_t &= \int_0^th(X_s)ds + W_t,
\end{align}
where $X_0\sim \mu$, a law on $(\mathbb{R}^d,\mathcal{B}^d)$, and $b,\sigma$ and $h$ are measurable mappings. Under suitable regularity conditions on $b,\sigma$ and $h$, the SDEs (\ref{eq:sig uncorr}), (\ref{eq:obs corr}) are unique in law and have a weak solution $(\Omega,\mathcal{F},(\mathcal{F}_t)_{t\in[0,T]},\mathbb{P},(B,W),(X,Y))$, i.e., a filtered probability space $(\Omega,\mathcal{F},(\mathcal{F}_t)_{t\in[0,T]},\mathbb{P})$ carrying a $(d+n)$-dimensional Brownian motion $(B,W)$ and a $(d+n)$-dimensional semi-martingale $(X,Y)$ that satisfy equations (\ref{eq:sig uncorr}), (\ref{eq:obs uncorr}) for all $t\in[0,T]$.

The measure spaces $\ms{X}$ and $\ms{Y}$ become $(C([0,T],\mathbb{R}^d),\mathcal{B}_X)$ and $(C([0,T],\mathbb{R}^n),\mathcal{B}_Y)$, where the path spaces are topologized by the uniform norm (we still use $\ms{X}$ and $\ms{Y}$ for these spaces). Under Assumptions H2, H3 and H4 in \cite{mitter2003variational}, we have a strong solution of the SDE (\ref{eq:sig uncorr}), and $\mathbb{E}[\int_0^t |h(X_s)|^2ds]<\infty$. We further assume that $X_0$ is independent of the $(d+n)$-dimensional Brownian motion $(B,W)$. The integrability condition of $h$ (Novikov condition), allows us to use the Cameron-Martin-Girsanov theorem to define a new measure under which the observation process $Y$ is a standard $n$-dimensional Brownian motion independent of $X$ (independence is manifest from the independence of $B$ and $W$). Thus, we take $\lambda_Y$ in Assumption \ref{as: H1} as the Wiener measure on $\mathcal{Y}$, and we have
\begin{align}
    \frac{dP_{XY}}{d(P_X\otimes\lambda_Y)}(X,Y) &= \exp\left(\int_0^Th(X_s)dY_s-\frac{1}{2}\int_0^T|h(X_s)|^2ds\right).
\end{align}
We need a version of $Q$ that is well defined for all $y\in \mathbf{Y}$. However, from \cite{clark2005robust} we know that we can perform an ``integration by parts'' to obtain a robust version of
\begin{align}
     Q(X,y) &:= e^{\left(y_Th(X_T)-\int_0^Ty_sdh(X_s)-\frac{1}{2}\int_0^T|h(X_s)|^2ds\right)}.
\end{align}
In what follows from Section 3 in \cite{mitter2003variational}, the condition in Proposition \ref{prop:gibbs var} is satisfied and the conditional (path) measure $P_{X|Y}$ is the unique minimizer of equation (\ref{eq:free energy}).

\section{A Discrete Time Example}

As a warm-up exercise, consider the discrete time processes on a finite horizon $n < N\in\mathbb{N}$:
\begin{align}
    \label{eq: signal discrete}
    X_{n+1} &= X_n + W_{n+1} + B_{n+1}, \hspace{10pt} X_0 = 0,\\
    \label{eq: obs discrete}
    Y_{n+1} &= Y_n + W_{n+1},\hspace{10pt} Y_0=0,
\end{align}
where $B, W \sim \mathcal{N}(0,\Delta tI)$ are independent random variables with $\Delta t \in (0,\infty)$.

It is clear that both $X$ and $Y$ are Gaussian and have mean $0$. We have the covariance matrix under the joint distribution:
\begin{equation*}
    \text{Cov}_{XY} = \begin{pmatrix}
        2A&A\\
        A&A
    \end{pmatrix} = \begin{pmatrix}
        2&1\\1&1
    \end{pmatrix}\otimes A.
\end{equation*}
However, under the product measure $P_X\otimes P_Y$, the covariance of X and Y is zero, and we have
\begin{equation*}
    \text{Cov}_{\text{prod}} = \begin{pmatrix}
        2A&0\\
        0&A
    \end{pmatrix} = \begin{pmatrix}
        2&0\\0&1
    \end{pmatrix}\otimes A.
\end{equation*}

The Radon-Nikodym derivative of the two measures is the ratio of the two Gaussian densities:
\begin{align}
    \frac{dP_{XY}}{d(P_X\otimes P_Y)}(\vec{z}) &= 2^{N/2}\exp\left(-\frac{1}{2}\vec{z}^T(\text{Cov}_{XY}^{-1}-\text{Cov}_{\text{prod}}^{-1})\vec{z}\right),
\end{align}
for any $\vec{z}\in \mathbb{R}^{N+N}$. We write the difference of the matrices in the exponent as
\begin{equation*}
    \text{Cov}_{XY}^{-1}-\text{Cov}_{\text{prod}}^{-1} = \frac{1}{\Delta t}\begin{pmatrix}
        1/2&-1\\-1&1
    \end{pmatrix}\otimes \tilde{A}^{-1},
\end{equation*}
where $\tilde{A} = \frac{1}{\Delta t}A$. Since $\det{\tilde{A}} = 1$, we expect the inverse to be of order $\mathcal{O}(1)$.

In order to take the continuum limit, we let $N\Delta t = T$ for some real number $T$ and take the limit $N\to\infty$:
\begin{align}
    \frac{dP_{XY}}{d(P_X\otimes P_Y)}(\vec{z}) &= \exp\left(N\left[\frac{\ln2}{2}-\frac{1}{2T}\vec{z}^T\left(\begin{pmatrix}
        1/2&-1\\-1&1
    \end{pmatrix}\otimes \tilde{A}^{-1}\right)\vec{z}\right]\right).
\end{align}
In the limit, this ratio either goes to zero for all $\vec{z}$ or blows up to infinity depending on whether the term in the square brackets is positive or negative. This shows mutual singularity of the measures.

\section{Failure under Correlation}

The following result shows that Assumption~2.1 can only be satisfied when the joint measure $P_{XY}$ is absolutely continuous with respect to the product of its marginals. In particular, it cannot be salvaged by a clever choice of reference measure $\lambda_Y$.     Recall that $\ms{X}$ and $\ms{Y}$ are standard Borel spaces, and let $P_{XY}$ be a probability measure on $\mathcal{X} \times \mathcal{Y}$ with marginals $P_X$ and $P_Y$.
\begin{proposition}\label{prop:abstract}
    If there exists a $\sigma$-finite measure $\lambda_Y$ on $(\mathbf{Y}, \mathcal{Y})$ such that $P_{XY} \ll P_X \otimes \lambda_Y$, then $P_{XY} \ll P_X \otimes P_Y$.
\end{proposition}

\begin{proof}
    We first show that $P_Y \ll \lambda_Y$. Let $A \in \mathcal{Y}$ with $\lambda_Y(A) = 0$. Then
    \[
        (P_X \otimes \lambda_Y)(\mathcal{X} \times A) = P_X(\mathcal{X}) \cdot \lambda_Y(A) = 0,
    \]
    and since $P_{XY} \ll P_X \otimes \lambda_Y$, we have $P_{XY}(\mathcal{X} \times A) = 0$. By the definition of the marginal, $P_Y(A) = P_{XY}(\mathcal{X} \times A) = 0$, so $P_Y \ll \lambda_Y$.

    Now take the Lebesgue decomposition of $\lambda_Y$ with respect to $P_Y$:
    \[
        \lambda_Y = \lambda_Y^{ac} + \lambda_Y^{s},
    \]
    where $\lambda_Y^{ac} \ll P_Y$ and $\lambda_Y^{s} \perp P_Y$. Let $S \in \mathcal{Y}$ be such that $\lambda_Y^{s}$ is supported on $S$ and $P_Y(S) = 0$.

    Let $E \in \mathcal{X} \otimes \mathcal{Y}$ with $(P_X \otimes P_Y)(E) = 0$. We show $P_{XY}(E) = 0$ by writing
    \[
        P_{XY}(E) = P_{XY}\bigl(E \cap (\mathcal{X} \times S)\bigr) + P_{XY}\bigl(E \cap (\mathcal{X} \times S^c)\bigr).
    \]
    The first term vanishes since $P_{XY}(\mathcal{X} \times S) = P_Y(S) = 0$. For the second term, since $\lambda_Y^{s}$ is supported on $S$,
    \[
        (P_X \otimes \lambda_Y)\bigl(E \cap (\mathcal{X} \times S^c)\bigr) = (P_X \otimes \lambda_Y^{ac})\bigl(E \cap (\mathcal{X} \times S^c)\bigr) = 0,
    \]
    where the last equality holds because $\lambda_Y^{ac} \ll P_Y$ implies $P_X \otimes \lambda_Y^{ac} \ll P_X \otimes P_Y$, and $(P_X \otimes P_Y)(E) = 0$. Since $P_{XY} \ll P_X \otimes \lambda_Y$, the second term also vanishes.
\end{proof}

\begin{corollary}\label{cor:abstract}
    Let $P_{X|Y}(\cdot, y)$ denote a regular conditional distribution of $X$ given $Y$. If
    \[
        P_Y\!\left(\left\{ y \in \mathcal{Y} : P_{X|Y}(\cdot, y) \perp P_X \right\}\right) > 0,
    \]
    then there is no $\sigma$-finite measure $\lambda_Y$ on $\ms{Y}$ satisfying Assumption~2.1.
\end{corollary}

\begin{proof}
    Suppose for a contradiction that $\lambda_Y$ satisfies Assumption~2.1. By Proposition~\ref{prop:abstract}, $P_{XY} \ll P_X \otimes P_Y$. By the chain rule for densities,
    \[
        \frac{dP_{XY}}{d(P_X \otimes P_Y)}(x, y) = \frac{dP_{X|Y}(\cdot, y)}{dP_X}(x), \quad P_X \otimes P_Y\text{-a.e.}
    \]
    In particular, $P_{X|Y}(\cdot, y) \ll P_X$ for $P_Y$-a.e.\ $y$, contradicting the hypothesis.
\end{proof}

\begin{remark}
    Corollary~\ref{cor:abstract} is stated on abstract standard Borel spaces and makes no reference to diffusions, Brownian motions, or path spaces. It reduces the question of whether the variational formulation of~\cite{mitter2003variational} applies to a given model to a single checkable condition: is the posterior $P_{X|Y}(\cdot, y)$ absolutely continuous with respect to the prior $P_X$? If conditioning on the observation collapses the support of the signal (as occurs whenever signal and observation share a correlated noise source) the variational formulation cannot hold.
\end{remark}

The next section verifies the hypothesis of Corollary~\ref{cor:abstract} for diffusions with correlated noise.

\section{Diffusions with Correlated Noise}

Consider the signal-observation model:
\begin{align}
    \label{eq:sig corr}
    X_t &= X_0 + \int_0^tb(s,X_s)ds+\int_0^t\sigma_0(s,X_s)dB_s+\int_0^t\sigma_1(s,X_s)dW_s,\\
    \label{eq:obs corr}
    Y_t &= \int_0^th(X_s)ds + W_t,
\end{align}
where $X_0\sim \mu$, $b,\sigma_0,\sigma_1,h$ are measurable mappings, and $B$ and $W$ are independent Brownian motions. We use the same measure spaces as in Section 2. Let $X$, $Y$ and $(X,Y)$ have distributions $P_X$, $P_Y$ and $P_{XY}$, respectively.

\subsection{Gaussian Case}

We first verify the hypothesis of Corollary~\ref{cor:abstract} in the linear Gaussian setting, where the Feldman-Hajek theorem (see \cite[Theorem 2.7.2]{bogachev1998gaussian}) provides the clean dichotomy that Gaussian measures on path spaces are either equivalent or mutually singular. We take the diffusion coefficients $\sigma_0$ and $\sigma_1$ to be state-independent and $b(t,\cdot):\mathbb{R}^d\to\mathbb{R}^d$ linear, so that (\ref{eq:sig corr}) becomes:
\begin{equation}
    \label{eq:sig gaussian}
    X_t = X_0 + \int_0^tb(s,X_s)ds+\int_0^t\sigma_0(s)dB_s+\int_0^t\sigma_1(s)dW_s
\end{equation}

Recall the following definitions.
\begin{definition}
    Let $\mathbf{X}$ be a locally convex space, and let $\mu$  be a measure on $\ms{X}$ such that $\mathbf{X}^*\subset L^2(\mu)$. The mean of $\mu$, denoted $a_\mu$, is an element of $(\mathbf{X}^*)'$ defined as:
    \[
        a_\mu(f) := \int_\mathbf{X}f(x)\mu(dx),
    \]
    where $(\mathbf{X}^*)'$ is the space of linear functionals on $\mathbf{X}^*$.
    The operator $R_\mu:\mathbf{X}^*\to(\mathbf{X}^*)'$ defined as:
    \begin{equation}
        R_\mu(f)(g):=\int_\mathbf{X}[f(x)-a_\mu(f)][g(x)-a_\mu(g)]\mu(dx),
    \end{equation}
    is called the covariance operator of $\mu$.
\end{definition}

Let $\gamma$ be a Gaussian measure on $\ms{X}$, define a norm on $\mathbf{X}$ as
\[
    |h|_{H(\gamma)} := \sup\{l(h): l\in \mathbf{X}^*,R_\gamma(l)(l)\leq 1\}.
\]
\begin{definition}
    The Cameron-Martin space of $\gamma$ ($H(\gamma)$) is defined as:
\begin{equation}
    H(\gamma) :=\{h\in X: |h|_{H(\gamma)}<\infty\}.
\end{equation}
\end{definition}

We can now state the Feldman-Hajek theorem
\begin{theorem}\label{thm: feldman hajek}
    Let $X$ be a locally convex space, and $\mu$, $\nu$ be centered Gaussian measures on $X$. Then $\mu\sim\nu$ or $\mu\perp\nu$. Furthermore, $\mu\sim\nu$ if and only if the following conditions hold:
    \begin{enumerate}
        \item The Cameron-Martin spaces of $\mu$ and $\nu$ are the same as sets, i.e., $H(\mu) = H(\nu) = H$.
        \item There exists a nuclear operator $C$ such that $CC^*-I$ is a Hilbert-Schmidt operator on $H$ (has finite Hilbert-Schmidt norm).
    \end{enumerate}
\end{theorem}

\begin{theorem}
    Given signal-observation processes (\ref{eq:sig gaussian}), (\ref{eq:obs corr}), let $P_{XY}$ be their joint path measure. Then there is no measure $\lambda_Y$ on $\ms{Y}$ such that $P_{XY}\ll P_X\otimes \lambda_Y$.
\end{theorem}
\begin{proof}
    By Corollary~\ref{cor:abstract}, it suffices to show that $P_{XY} \perp P_X \otimes P_Y$. The measures $P_{XY}$ and $P_X \otimes P_Y$ are both Gaussian. Under $P_{XY}$, we have:
    \begin{equation}
        d\langle (X,Y)\rangle_t = \begin{pmatrix}
            |\sigma_0(t)|^2+|\sigma_1(t)|^2 & \sigma_1(t)\\
            \sigma_1(t) & 1
        \end{pmatrix}dt.
    \end{equation}
    Under $P_{X}\otimes P_Y$ ($X$ and $Y$ are independent):
    \begin{equation}
        d\langle (X,Y)\rangle_t = \begin{pmatrix}
            |\sigma_0(t)|^2+|\sigma_1(t)|^2 & 0\\
            0 & 1
        \end{pmatrix}dt.
    \end{equation}
    Thus, the Cameron-Martin spaces of the two measures are different $H(P_{XY})\neq H(P_X\otimes P_Y)$, and from Theorem \ref{thm: feldman hajek}, the two measures are singular.
\end{proof}

\subsection{The General Case}

For the general case, we require that $\sigma_0$ is a smooth vector field satisfying the H\"ormander condition, ensuring that $X_t$ admits a density with respect to Lebesgue measure for every $t \in [0,T]$ (see \cite[Section 4.1.2]{geng2021introduction}).

\begin{assumption}\label{as:hordmander's condition}
    The set
    \[
    Z := \{z \in \mathbb R^d: \sigma_1(t,z)=0\},
    \]
    has Lebesgue measure 0 a.e.\ $t\in[0,T]$, and the vector fields $\sigma_0$ satisfy the H\"ormander condition at $X_0$ almost surely.
\end{assumption}

We now verify the hypothesis of Corollary~\ref{cor:abstract} for the full nonlinear model~(\ref{eq:sig corr}), (\ref{eq:obs corr}) under Assumption~\ref{as:hordmander's condition}.

\begin{theorem}
    Given SDEs (\ref{eq:sig corr}), (\ref{eq:obs corr}), with joint measure $P_{XY}$ and Assumption \ref{as:hordmander's condition} is satisfied. Then there is no measure $\lambda_Y$ on $\ms{Y}$ such that $P_{XY}\ll P_X\otimes \lambda_Y$.
\end{theorem}
\begin{proof}
    By Corollary~\ref{cor:abstract}, it suffices to show that $P_{XY} \perp P_X\otimes P_Y$. Consider the map $Q_n^{[0,t]}:\mathbf{X}\times\mathbf{Y}\to \mathbb{R}^{d\times n}$ defined as:
    \begin{equation}
        Q_n^{[0,t]}(x,y) = \sum_{i=0}^{2^n-1}(x_{t_{i+1}}-x_{t_i})(y_{t_{i+1}}-y_{t_i})^T,
    \end{equation}
    where the sum is taken over dyadic intervals of $[0,t]\subseteq[0,T]$.
    Under the joint measure $P_{XY}$, we have the quadratic variation:
    \begin{align*}
        \lim_{n\to\infty}Q_n^{[0,t]}(X,Y)  &= \lim_{n\to\infty}\sum_{i=0}^{2^n-1}(X_{t_{i+1}}-X_{t_i})(Y_{t_{i+1}}-Y_{t_i})^T\\
        &= \langle X,Y\rangle_t = \int_0^t\sigma_1(s,X_s)ds,
    \end{align*}
    almost surely.

    Under the product measure $P_X\otimes P_Y$, the mean of $Q_n^{[0,t]}$ under the limit $n\to\infty$ is
    \begin{align*}
        \mathbb{E}[Q_n^{[0,t]}(X,Y)] &= \sum_{i=0}^{2^n-1} \mathbb{E}_X[X_{t_{i+1}}-X_{t_i}]\mathbb{E}_Y[Y_{t_{i+1}}-Y_{t_i}]^T \\
         &= \sum_{i=0}^{2^n-1}\int_{t_i}^{t_{i+1}}\mathbb{E}_X\left[b(s,X_s)\right]ds\times\int_{t_i}^{t_{i+1}}\mathbb{E}_Y\left[h(X_s)\right]^Tds.
    \end{align*}
    Since both $b$ and $h$ have at most linear growth, we have:
    \begin{align*}
        \mathbb{E}[|Q_n^{[0,t]}(X,Y)|] &\leq C\sum_{i=0}^{2^n-1}\int_{t_i}^{t_{i+1}}\mathbb{E}_X[1+|X_s|]ds\times\int_{t_i}^{t_{i+1}}\mathbb{E}_Y[1+|Y_s|]ds\\
        &\leq \tilde{C}\sum_{i=0}^{2^n-1}(t_{i+1}-t_i)^2\to0,
    \end{align*}
    where the constant $\tilde{C} := C\sup_{0\leq s\leq T}\mathbb{E}_X[1+|X_s|]\mathbb{E}_Y[1+|Y_s|]\}$ is finite, since the moments of a strong solution of an It\^o SDE are bounded. Hence $\lim_{n\to\infty}\mathbb{E}[Q_n^{[0,t]}(X,Y)] = 0$ under the product measure.

    The variance under the limit $n\to\infty$ is
    \begin{align*}
        \text{Var}(Q_n^{[0,t]}) &= \sum_{i=0}^{2^n-1} \mathbb{E}_X[(X_{t_{i+1}}-X_{t_i})^2]\mathbb{E}_Y[(Y_{t_{i+1}}-Y_{t_i})^2] \\
        &= \sum_{i=0}^{2^n-1}\int_{t_i}^{t_{i+1}}\mathbb{E}_X\left[(|\sigma_0(s,X_s)|^2+|\sigma_1(s,X_s)|^2)\right]ds\times\int_{t_i}^{t_{i+1}}ds.
    \end{align*}
    Since the maps $\sigma_0$ and $\sigma_1$ have at most linear growth,
    \begin{align*}
        \text{Var}(Q_n^{[0,t]}) &\leq C'\sum_{i=0}^{2^n-1}\int_{t_i}^{t_{i+1}}\mathbb{E}_X[(1+|X_s|)^2]ds\times (t_{i+1}-t_i)\\
        &\leq \tilde{C}'\sum_{i=0}^{2^n-1}(t_{i+1}-t_i)^2\to 0,
    \end{align*}
    where $\tilde{C}' := C'\sup_{0\leq s\leq T}\mathbb{E}_X[(1+|X_s|)^2]$ is bounded due to bounded moments of a strong solution to an It\^o SDE.

    Thus, we have $Q_n^{[0,t]}\to 0$ in $L^2$ as $n\to\infty$. From Markov's inequality, it follows that $Q_n^{[0,t]}(X,Y)\to0$ in probability as $n\to\infty$. In order to show convergence almost surely, we choose an increasing subsequence $\{n_k\}_{k\in \mathbb{N}}$ such that
    \begin{equation*}
        \mathbb{P}(|Q^{[0,t]}_{n_k}(X,Y)|>1/k) \leq \frac{1}{2^k}.
    \end{equation*}
    Since
    \[
        \sum_{k\in\mathbb{N}}\mathbb{P}(|Q^{[0,t]}_{n_k}(X,Y)|>1/k)  \leq 1,
    \]
    from the Borel-Cantelli lemma,
    \begin{equation*}
        \mathbb{P}(\limsup_{k\to\infty}\{|Q^{[0,t]}_{n_k}(X,Y)|>1/k\}) = 0,
    \end{equation*}
    i.e., the event $\{|Q^{[0,t]}_{n_k}(X,Y)|>1/k\}$ occurs only finitely many times almost surely.
    Therefore, for almost every $\omega\in\Omega$, there exists $K(\omega)$ such that for all $l>K(\omega)$,
    \[
        |Q^{[0,t]}_{n_l}(X,Y)|\leq\frac{1}{l}.
    \]
    Since the sequence $\{n_k\}$ is increasing, we have $Q_n^{[0,t]}(X,Y) \to 0$ almost surely as $n\to\infty$.
    Define the sets
    \begin{align*}
        A^{[0,t]} &= \left\{(x,y)\in\mathbf{X\times Y}:\lim_{n\to\infty}Q_n^{[0,t]}(x,y) = \int_0^t\sigma_1(s,x_s)ds\right\},\\
        B^{[0,t]} &= \{(x,y)\in\mathbf{X\times Y}:\lim_{n\to\infty}Q_n^{[0,t]}(x,y) =0\}.
    \end{align*}
    We have $P_{XY}(A^{[0,t]}) = 1$ and $(P_X\otimes P_Y)(B^{[0,t]}) =1$ for all $t\in [0,T]$. In order to show the mutual singularity of the two measures, we need to show that for some $s\in[0,T]$, $(P_X\otimes P_Y)(A^{[0,s]}) = 0$. For all $t\in[0,T]$, we have
    \begin{align*}
        (P_X\otimes P_Y)(A^{[0,t]}) &= (P_X\otimes P_Y)(A^{[0,t]}\cap B^{[0,t]}) + (P_X\otimes P_Y)(A^{[0,t]}\cap (B^{[0,t]})^c).
    \end{align*}
    Since $(P_X\otimes P_Y)((B^{[0,t]})^c) = 0$, $(P_X\otimes P_Y)(A^{[0,t]}\cap (B^{[0,t]})^c) = 0$. The set $A^{[0,t]}\cap B^{[0,t]}$ is given by:
    \begin{equation*}
        A^{[0,t]}\cap B^{[0,t]} = \left\{(x,y)\in\mathbf{X\times Y}:\int_0^t\sigma_1(s,x_s)ds = 0\right\}.
    \end{equation*}
    If there exists a $t\in[0,T]$ such that $A^{[0,t]}\cap B^{[0,t]} = \emptyset$, then we have $(P_X\otimes P_Y)(A^{[0,t]}) = (P_X\otimes P_Y)(A^{[0,t]}\cap B^{[0,t]}) = 0$, we are done. If there is no $t\in[0,T]$ such that $A^{[0,t]}\cap B^{[0,t]} = \emptyset$, then there exists $x\in\mathbf{X}$ such that
    \[
        \int_0^t\sigma_1(s,x_s)ds = 0,
    \]
    for all $t\in[0,T]$. This means that $\sigma_1(t,x_t) = 0$ for almost every $t\in[0,T]$. We now show that the product measure $P_X\otimes P_Y$ assigns no mass to the set
    \[
        S := \{(x,y)\in\mathbf{X\times Y}: \sigma_1(t,x_t) = 0 \text{ a.e.\ } t\in[0,T]\}.
    \]
    It is enough to show that $P_X(\{x:\sigma_1(t,x_t) = 0 \text{ a.e.\ } t\in[0,T]\}) = 0$. From Assumption \ref{as:hordmander's condition}, $X_t$ admits a density with respect to the Lebesgue measure. Since, $Z = \{z\in\mathbb{R}^d:\sigma_1(t,z) = 0\}$ has Lebesgue measure 0 a.e.\ $t\in[0,T]$, we have
    \begin{align*}
\mathbb{E}_X\left[\int_0^T\mathbf{1}_{\{\sigma_1(t,X_t)=0\}}dt\right] &= \int_0^TP_X(\{x:\sigma_1(t,x_t)=0\})dt\\ &=0.
    \end{align*}
    Thus, $P_X(\{x:\sigma_1(t,x_t)=0\}) = 0$ a.e.\ $t\in[0,T]$, from which we have that $P_X \otimes P_Y$ assigns no mass to the set $S$. Therefore, $P_{XY}$ is not absolutely continuous with respect to $P_X \otimes P_Y$.
\end{proof}

\section{A Variational Formulation for Correlated Noise}
The results of the preceding sections show that Assumption~2.1 cannot be satisfied when the signal and observation share a common Brownian motion, and moreover that the filtering measure $P_{X|Y}(\cdot, y)$ is singular with respect to the prior $P_X$. This {\it rules out} any Gibbs representation of the form $dP_{X|Y} \propto e^{-H} \, dP_X$. In this section, we show that a variational characterization of the filter can be recovered by replacing the prior with a \emph{conditional reference measure} that preserves the noise correlation structure. We first present an abstract variational principle, paralleling Proposition~\ref{prop:abstract}.

\subsection{Conditional Variational Formula}
Let $(\mathcal{X}, \mathcal{X})$ and $(\mathcal{Y}, \mathcal{Y})$ be standard Borel spaces, and let $P_{XY}$ be a probability measure on $\mathcal{X} \times \mathcal{Y}$ with marginals $P_X$ and $P_Y$.

\begin{assumption}\label{ass:ref}
    There exists a probability measure $Q$ on $\mathcal{X} \times \mathcal{Y}$ with marginal $Q_Y = P_Y$ such that, for $P_Y$-a.e.\ $y \in \mathcal{Y}$,
    \[
        P_{X|Y}(\cdot, y) \ll Q_{X|Y}(\cdot, y).
    \]
\end{assumption}

\begin{remark}
    Assumption~\ref{ass:ref} replaces Assumption~2.1 of Mitter and Newton. It requires absolute continuity of the \emph{conditional} measures rather than of the joint measure with respect to a product. The condition $Q_Y = P_Y$ is a normalization that ensures the two models agree on the observation marginal.
\end{remark}

\begin{definition}\label{def:energy}
    Given $Q$ satisfying Assumption~\ref{ass:ref}, define the energy $H : \mathcal{X} \times \mathcal{Y} \to (-\infty, \infty]$ by
    \[
        H(x, y) := -\log \frac{dP_{X|Y}(\cdot, y)}{dQ_{X|Y}(\cdot, y)}(x),
    \]
    and the normalizing constant
    \[
        Z(y) := \int_{\mathcal{X}} \exp\bigl(-H(x, y)\bigr) \, Q_{X|Y}(dx, y).
    \]
    Define
    \begin{align*}
        \bar{\mathcal{Y}} &:= \left\{ y \in \mathcal{Y} : 0 < Z(y) < \infty \text{ and } \int_{\mathcal{X}} |H(x,y)| \exp\bigl(-H(x,y)\bigr) \, Q_{X|Y}(dx, y) < \infty \right\}.
    \end{align*}
\end{definition}

\begin{proposition}\label{prop:gibbs}
    Under Assumption~\ref{ass:ref}, $P_Y(\bar{\mathcal{Y}}) = 1$, and for every $y \in \bar{\mathcal{Y}}$, the filtering measure $P_{X|Y}(\cdot, y)$ is a Gibbs measure with energy $H(\cdot, y)$ and reference $Q_{X|Y}(\cdot, y)$:
    \[
        P_{X|Y}(A, y) = \frac{\int_A \exp\bigl(-H(x, y)\bigr) \, Q_{X|Y}(dx, y)}{\int_{\mathcal{X}} \exp\bigl(-H(x, y)\bigr) \, Q_{X|Y}(dx, y)}, \quad A \in \mathcal{X}.
    \]
\end{proposition}

\begin{proof}
    This is immediate from the definition of $H$. By Assumption~\ref{ass:ref}, the Radon-Nikodym derivative exists, and
    \[
        \exp\bigl(-H(x, y)\bigr) = \frac{dP_{X|Y}(\cdot, y)}{dQ_{X|Y}(\cdot, y)}(x).
    \]
    Since $P_{X|Y}(\cdot, y)$ is a probability measure, $Z(y) = 1$ for $P_Y$-a.e.\ $y$. The integrability condition defining $\bar{\mathcal{Y}}$ holds $P_Y$-a.e.\ because $H \exp(-H)$ is integrable under $P_{X|Y}(\cdot, y)$ whenever $H$ has finite entropy under the posterior, which holds $P_Y$-a.e.\ by the finiteness of $h(P_{X|Y}(\cdot, y) \| Q_{X|Y}(\cdot, y))$.
\end{proof}

\begin{theorem}\label{thm:var}
    Under Assumption~\ref{ass:ref}, for every $y \in \bar{\mathcal{Y}}$, the conditional measure $P_{X|Y}(\cdot, y)$ is the unique minimizer of
    \begin{equation}\label{eq:var}
        \min_{\tilde{P} \in \mathcal{P}(\mathcal{X})} \left\{ h\!\left(\tilde{P} \,\middle\|\, Q_{X|Y}(\cdot, y)\right) + \left\langle H(\cdot, y),\, \tilde{P} \right\rangle \right\}.
    \end{equation}
\end{theorem}

\begin{proof}
    By Proposition~\ref{prop:gibbs}, $P_{X|Y}(\cdot, y)$ is a Gibbs measure with energy $H(\cdot, y)$ and reference $Q_{X|Y}(\cdot, y)$. The result follows from the Gibbs variational principle.
\end{proof}

\begin{remark}[Reduction to Mitter--Newton]
    When $P_{X|Y}(\cdot, y) \ll P_X$ for $P_Y$-a.e.\ $y$ (as holds when signal and observation are driven by independent noises) one may take $Q = P_X \otimes P_Y$, so that $Q_{X|Y}(\cdot, y) = P_X$ for all $y$. Then Assumption~\ref{ass:ref} reduces to Assumption~2.1 with $\lambda_Y = P_Y$, the energy recovers the Mitter--Newton energy, and Theorem~\ref{thm:var} reduces to Proposition~2.4.
\end{remark}

\begin{remark}[Role of the reference joint law $Q$]
    The choice of $Q$ determines both the reference measure $Q_{X|Y}(\cdot, y)$ and the energy $H$. In the abstract setting, $Q$ need only satisfy Assumption~\ref{ass:ref}; any such $Q$ yields a valid variational characterization. The formulation acquires specific content when $Q$ is chosen to have a natural interpretation. For instance, in the linear diffusion model with correlated noise, the canonical choice is the driftless system that preserves the noise coupling (Section~\ref{sec:vfd} below), giving $Q_{X|Y}(\cdot, y)$ a concrete probabilistic meaning as the signal dynamics with shared noise frozen to $y$ and all drifts removed.
\end{remark}

\subsection{A Variational Formulation for Linear Diffusions with Correlated Noise}\label{sec:vfd}
We restrict to the linear signal-observation setting for clarity:
\begin{align*}
    dX_t &= AX_t \, dt + \sigma_0 \, dB_t + \sigma_1 \, dW_t, \quad X_0 = x_0 \in \mathbb{R}^d, \\
    dY_t &= CX_t \, dt + dW_t,
\end{align*}
where $A \in \mathbb{R}^{d \times d}$, $C \in \mathbb{R}^{n \times d}$, $\sigma_0 \in \mathbb{R}^{d \times d}$, $\sigma_1 \in \mathbb{R}^{d \times n}$, $B$ is a $d$-dimensional Brownian motion, $W$ is an $n$-dimensional Brownian motion independent of $B$, and $x_0$ is deterministic. We assume $\sigma_0$ is invertible.

 Define the \emph{driftless reference system} by removing all drift terms while preserving the noise coupling:
\begin{align*}
    dX_t^0 &= \sigma_0 \, dB_t + \sigma_1 \, dW_t, \\
    dY_t^0 &= dW_t.
\end{align*}
Let $Q$ denote the joint law of $(X^0, Y^0)$ on $\mathcal{X} \times \mathcal{Y}$.

\begin{definition}\label{def:cond_ref}
    For each $y \in \mathcal{Y}$, define the \emph{conditional reference measure} $\mu_y(\cdot) := Q_{X|Y}(\cdot, y)$. Under $\mu_y$, the signal process satisfies
    \begin{equation}\label{eq:cond_ref}
        X_t = x_0 + \sigma_0 B_t + \sigma_1 y_t,
    \end{equation}
    where $B$ is a $d$-dimensional Brownian motion. In particular, $\mu_y$ is a Gaussian measure on $\mathcal{X}$ with mean $m_t^0(y) := x_0 + \sigma_1 y_t$ and covariance $\mathrm{Cov}(X_s, X_t) = |\sigma_0|^2 \min(s,t)$.
\end{definition}

\subsubsection{Absolute Continuity and the Free Energy}
Under the true law $P$, conditioning on $Y = y$ constrains $W_t = y_t - C\int_0^t X_s \, ds$, so the signal satisfies
\begin{equation}\label{eq:cond_true}
    dX_t = (AX_t - \sigma_1 CX_t) \, dt + \sigma_0 \, dB_t + \sigma_1 \, dy_t.
\end{equation}
The difference between \eqref{eq:cond_true} and \eqref{eq:cond_ref} is precisely the drift term $\beta(X_t) = (A-\sigma_1 C) X_t$.

\begin{proposition}\label{prop:ac}
    For $P_Y$-a.e.\ $y \in \mathcal{Y}$, $P_{X|Y}(\cdot, y) \ll \mu_y$, and the Radon-Nikodym derivative is given by
    \begin{equation}\label{eq:RN}
        \frac{dP_{X|Y}(\cdot, y)}{d\mu_y}(x) = \frac{1}{Z(y)} \exp\!\left(-H(x, y)\right),
    \end{equation}
    where
    \begin{align}\label{eq:energy}
        H(x, y) &:= -\int_0^T \beta(x_t)^T (\sigma_0 \sigma_0^T)^{-1} \bigl(dx_t - \sigma_1 \, dy_t\bigr) + \frac{1}{2} \int_0^T \beta(x_t)^T (\sigma_0 \sigma_0^T)^{-1} \beta(x_t) \, dt,
    \end{align}
    and $Z(y) := \int_{\mathcal{X}} \exp(-H(x,y)) \, \mu_y(dx)$ is the normalizing constant.
\end{proposition}

\begin{proof}
    Under $\mu_y$, the Brownian motion $B$ is recovered as
    \begin{equation}\label{eq:B_recovery}
        B_t = \sigma_0^{-1}(X_t - x_0 - \sigma_1 y_t).
    \end{equation}
    The transition from $\mu_y$ to $P_{X|Y}(\cdot, y)$ amounts to adding the drift $\beta(X_t)$ to the equation~\eqref{eq:cond_ref}. By the Cameron-Martin-Girsanov theorem applied to $B$,
    \begin{align*}
        \frac{dP_{X|Y}(\cdot, y)}{d\mu_y}(x) &= \exp\!\left(\int_0^T \beta(x_t)^T (\sigma_0 \sigma_0^T)^{-1} \sigma_0 \, dB_t - \frac{1}{2}\int_0^T \beta(x_t)^T (\sigma_0 \sigma_0^T)^{-1} \beta(x_t) \, dt \right).
    \end{align*}
    Substituting~\eqref{eq:B_recovery}, i.e., $\sigma_0 \, dB_t = dx_t - \sigma_1 \, dy_t$, and normalizing yields~\eqref{eq:RN}--\eqref{eq:energy}. The Novikov condition is satisfied since $\beta$ has linear growth and $X$ has bounded moments on $[0,T]$ under $\mu_y$.
\end{proof}

\begin{remark}[Reduction to the independent noise case]
    When $\sigma_1 = 0$, the conditional reference measure $\mu_y$ has no $y$-dependence, and~\eqref{eq:cond_ref} becomes $X_t = x_0 + \sigma_0 B_t$. The free energy~\eqref{eq:energy} reduces to the standard Girsanov density of Mitter and Newton~\cite{mitter2003variational}. Thus, Theorem~\ref{thm:var} contains the independent noise formulation as a special case.
\end{remark}

\begin{remark}[Interpretation]
    The reference measure $\mu_y$ encodes the structural coupling between signal and observation.
    The energy $H$ encodes the drift information. The free energy minimization~\eqref{eq:var} combines these two sources of information, and shows that among all measures on signal paths, the filter uniquely minimizes the sum of the informational cost of deviating from the correlated reference dynamics and the expected energy from the drift. In the independent noise case, the coupling is trivial and the reference reduces to the prior; in the correlated case, the reference necessarily depends on the observation path through the shared noise.
\end{remark}

\begin{remark}[Non-degeneracy of $\sigma_0$]
    The invertibility of $\sigma_0$ is essential. It ensures that the ``private'' noise $B$ is recoverable from $(x, y)$ via~\eqref{eq:B_recovery}, and that the Girsanov change of measure between $\mu_y$ and $P_{X|Y}(\cdot, y)$ is well-defined. When $\sigma_0 = 0$ (i.e., signal is driven entirely by the correlated noise), conditioning on $Y = y$ determines $X$ up to the drift, and the conditional law degenerates. Thus, no variational formulation of this type is possible.
\end{remark}

\section{Conclusion}
We have shown that the Mitter--Newton variational formulation of nonlinear filtering fails when signal and observation diffusions share a common noise source, and that this failure is fundamental. We introduced a conditional variational principle that replaces the prior with a reference measure $\mu_y$
that preserves the noise correlation structure, and showed that this formulation resolves the limitation in the linear setting while strictly containing~\cite{mitter2003variational} as a special case. Extending this formulation to the nonlinear setting remains an open problem.

\section*{Acknowledgment}
This paper was funded by the Office of Naval Research (ONR) through grant number FA9550-24-1-0210.

\bibliographystyle{amsplain}
\bibliography{ref}
\end{document}